\documentclass[letterpaper, 10 pt, conference]{ieeeconf}  

\IEEEoverridecommandlockouts                              

\usepackage{epsfig} 
\usepackage{amsmath} 
\usepackage{amssymb}  

\newtheorem{thm}{Theorem}

\newtheorem{cor}{Corollary}

\newtheorem{defn}{Definition}
\newtheorem{rem}{Remark}
\newtheorem{example}{Example}

\DeclareMathOperator{\Rank}{rank}

\usepackage{comment}

\title{\LARGE \bf
A Trajectory-Based Approach to Discrete-Time Flatness
}

\author{Johannes Diwold$^{1}$, Bernd Kolar$^{2}$ and Markus Sch{\"o}berl$^{1}$
	\thanks{*This work has been supported by the Austrian Science Fund (FWF) under grant number P~32151 and P~29964.}
	\thanks{$^{1}$Institute of Automatic Control and Control Systems Technology, Johannes Kepler University Linz, Altenbergerstra\ss e 66, 4040 Linz, Austria\newline
		{\tt\small johannes.diwold@jku.at,markus.schoeberl@jku.at}}%
	\thanks{$^{2}$Magna Powertrain Engineering Center Steyr GmbH \& Co KG, Steyrer Str. 32, 4300 St. Valentin, Austria
		{\tt\small bernd\_kolar@ifac-mail.org}}%
}

\begin{document}

\maketitle
\thispagestyle{empty}
\pagestyle{empty}




	



\setlength{\arraycolsep}{2pt} 
\begin{abstract}
For discrete-time systems, flatness is usually defined by replacing
the time-derivatives of the well-known continuous-time definition
by forward-shifts. With this definition, the class of flat systems
corresponds exactly to the class of systems which can be linearized
by a discrete-time endogenous dynamic feedback as it is proposed in
the literature. Recently, verifiable necessary and sufficient differential-geometric
conditions for this property have been derived. In the present contribution,
we make an attempt to take into account also backward-shifts. This
extended approach is motivated by the one-to-one correspondence of
solutions of flat systems to solutions of a trivial system as it is
known from the continuous-time case. If we transfer this idea to the
discrete-time case, this leads to an approach which also allows backward-shifts.
To distinguish the classical definition with forward-shifts and the
approach of the present paper, we refer to the former as forward-flatness.
We show that flat systems (in the extended sense with backward-shifts)
still share many beneficial properties of forward-flat systems. In
particular, they still are reachable/controllable, allow a straightforward
planning of trajectories and can be linearized by a certain subclass
of dynamic feedbacks.

\end{abstract}

\section{Introduction}

In the 1990s, the concept of flatness has been introduced by Fliess,
L\'{e}vine, Martin and Rouchon for nonlinear continuous-time systems
(see e.g. \cite{FliessLevineMartinRouchon:1995} and \cite{FliessLevineMartinRouchon:1999}).
Flat continuous-time systems have the characteristic feature that
all system variables can be parameterized by a flat output and its
time derivatives. This leads to a one-to-one correspondence of solutions
of a flat system to solutions of a trivial system with the same number
of inputs. Flat systems form an extension of the class of static feedback
linearizable systems and can be linearized by an endogenous dynamic
feedback. Their popularity stems from the fact that many physical
systems possess the property of flatness and that the knowledge of
a flat output allows an elegant solution to motion planning problems
and design of tracking controllers.

For nonlinear discrete-time systems, flatness is usually defined by
replacing the time-derivatives of the well-known continuous-time definition
by forward-shifts. More precisely, the flat output is a function of
the state variables, input variables, and forward-shifts of the input
variables. Conversely, the state- and input variables can be expressed
as functions of the flat output and its forward-shifts. This point
of view has been adopted in \cite{KaldmaeKotta:2013}, \cite{Sira-RamirezAgrawal:2004}
and \cite{KolarKaldmaeSchoberlKottaSchlacher:2016}. With this definition,
the class of flat systems corresponds exactly to the class of systems
which can be linearized by a discrete-time endogenous dynamic feedback
as it is proposed e.g. in \cite{Aranda-BricaireMoog:2008}. Recently,
verifiable necessary and sufficient differential-geometric conditions
have been derived in \cite{KolarSchoberlDiwold:2019} and \cite{KolarDiwoldSchoberl:2019}.
Furthermore, in \cite{DiwoldKolarSchoberl:2020} it has been shown
that in the two-input case even a transformation into a certain normal
form is always possible.

In this contribution, we focus solely on the one-to-one correspondence
of solutions of flat systems to solutions of a trivial system (arbitrary
trajectories that need not satisfy any equation) with the same number
of inputs, as it is known from the continuous-time case. For discrete-time
systems, this would mean that the flat output may depend both on forward-
and backward-shifts of the system variables. Conversely, the state-
and input variables could be expressed as functions of both forward-
and backward-shifts of the flat output. To distinguish the usual definition
of \cite{KaldmaeKotta:2013}, \cite{Sira-RamirezAgrawal:2004} and
\cite{KolarKaldmaeSchoberlKottaSchlacher:2016} with forward-shifts
from the alternative approach of the present paper, we refer to the
former as forward-flatness. A special case of this alternative definition
has already been suggested in \cite{GuillotMillerioux2020}, where
the flat output may depend also on backward-shifts of the input variables
but not on backward-shifts of the state variables. To justify our
alternative approach, we show that flat systems (in the extended sense
with backward-shifts) still share many beneficial properties of forward-flat
systems. In particular, we show that they still are reachable (and
hence controllable) and allow a straightforward planning of trajectories
to connect arbitrary points of the state space. Furthermore, we show
that they can be linearized by a dynamic feedback which shares the
beneficial properties of the class of continuous-time endogenous feedback.
With respect to the classical dynamic feedback linearization problem,
the following inclusions hold: (static feedback linearizable systems)
$\subset$ (forward-flat systems) $\subset$ (flat systems) $\subset$
(dynamic feedback linearizable systems). We show that for single-input
and linear systems the properties of flatness, forward-flatness and
static feedback linearizability are equivalent, and present an example
which shows that the class of forward-flat systems is a strict subset
of the class of flat systems in the extended sense.

The paper is organized as follows: In Section \ref{sec:Discrete-Time-Flatness}
we introduce the extended concept of flatness with both forward- and
backward-shifts and illustrate it by an example. Subsequently, we
discuss the special case of single-input systems. In Section \ref{sec:Feedback-Linearization}
we first demonstrate the planning of trajectories, prove the reachability
of flat systems and apply the concept to the sampled-data model of
an induction motor. Second, we show that flat systems can be linearized
by a particular subclass of dynamic feedbacks.\section{Discrete-Time Flatness with Forward- and Backward-Shifts} \label{sec:Discrete-Time-Flatness}Throughout
this contribution, we consider time-invariant discrete-time nonlinear
systems in state representation of the form
\begin{equation}
x^{i,+}=f^{i}(x,u)\,,\quad i=1,\dots,n\label{eq:System_Eq}
\end{equation}
with $\dim(x)=n$, $\dim(u)=m$ and smooth functions $f^{i}(x,u)$.
We assume that the systems meet the submersivity condition, i.e. that
the Jacobian-matrix of $f$ with respect to $(x,u)$ meets
\begin{equation}
\Rank(\partial_{(x,u)}f)=n\,.\label{eq:submersivity}
\end{equation}
This condition is necessary for reachability and consequently also
for flatness. However, we want to emphasize that we do not require
$\Rank(\partial_{x}f)=n$. As mentioned in \cite{Aranda-BricaireMoog:2008},
this property is always met by systems which stem from the exact or
approximate discretization of continuous-time systems. However, we
want to consider discrete-time systems in general, no matter whether
they stem from a discretization or not.

\subsection{Equivalence of Solutions\label{subsec:Equivalence-of-Solutions}}

To motivate our trajectory-based approach, we want to recall that
a continuous-time system $\dot{x}=f(x,u)$ is flat if there exists
a one-to-one correspondence between its solutions $(x(t),u(t))$ and
solutions $y(t)$ of a trivial system (sufficiently smooth but otherwise
arbitrary trajectories) with the same number of inputs (see e.g. \cite{FliessLevineMartinRouchon:1999}).

In the following, we attempt to define flatness for discrete-time
systems in exactly the same way. Within this paper, we call a discrete-time
system \eqref{eq:System_Eq} flat if there exists a one-to-one correspondence
between its solutions $(x(k),u(k))$ and solutions $y(k)$ of a trivial
system (arbitrary trajectories that need not satisfy any difference
equation) with the same number of inputs.

\[
\begin{array}{ccc}
x^{+}=f(x,u) &  & \begin{aligned}\text{trivial system}\end{aligned}
\\
\downarrow &  & \downarrow\\
(x(k),u(k)) & \overset{\text{one-to-one}}{\Longleftrightarrow} & (y(k))
\end{array}
\]
By one-to-one correspondence, we mean that the values of $x(k)$ and
$u(k)$ at some fixed time step $k$ may depend on an arbitrary but
finite number of future and past values of $y(k)$, i.e. on the whole
trajectory in an arbitrarily large but finite interval\footnote{Note that the time derivatives in the continuous-time case provide
	via the Taylor-expansion also information about the trajectory both
	in forward- and backward-direction.}. Conversely, the value of $y(k)$ at some fixed time step $k$ may
depend on an arbitrary but finite number of future and past values
of $x(k)$ and $u(k)$. Thus, the one-to-one correspondence of the
solutions can be expressed by maps of the form
\begin{equation}
(x(k),u(k))=F(k,y(k-r_{1}),\ldots,y(k),\ldots,y(k+r_{2}))\label{eq:parameterizing_map_forward_backward_k_variant}
\end{equation}
and
\begin{equation}
\begin{aligned}y(k)=\varphi(k,x(k-q_{1}),u(k-q_{1}),\ldots,\hphantom{aaaaaaaaaa}\\
x(k),u(k),\ldots,x(k+q_{2}),u(k+q_{2}))
\end{aligned}
\label{eq:flat_output_xu_redundant_k_variant}
\end{equation}
with suitable integers $r_{1},r_{2},q_{1},q_{2}$. These maps must
satisfy two conditions. First, in order to ensure the one-to-one correspondence,
the composition of \eqref{eq:parameterizing_map_forward_backward_k_variant}
with the occurring shifts of \eqref{eq:flat_output_xu_redundant_k_variant},
or vice versa, must yield the identity map. Second, since the trajectory
$y(k)$ of the trivial system is arbitrary, after substituting \eqref{eq:parameterizing_map_forward_backward_k_variant}
into the system equations \eqref{eq:System_Eq} they also must be
satisfied identically. Because of the time-invariance of the system
\eqref{eq:System_Eq}, within this paper we only consider maps
\begin{equation}
(x(k),u(k))=F(y(k-r_{1}),\ldots,y(k),\ldots,y(k+r_{2}))\label{eq:parameterizing_map_forward_backward}
\end{equation}
and
\begin{equation}
\begin{aligned}y(k)=\varphi(x(k-q_{1}),u(k-q_{1}),\ldots,x(k),u(k),\ldots,\\
x(k+q_{2}),u(k+q_{2}))\,.
\end{aligned}
\label{eq:flat_output_xu_redundant}
\end{equation}
which do not depend explicitly on the time step $k$.
\begin{rem}
	\label{rem:individual_shifts}The number of forward- and backward-shifts
	in (\ref{eq:parameterizing_map_forward_backward}) and (\ref{eq:flat_output_xu_redundant})
	can of course be different for the individual components of $y$,
	$x$ and $u$. Thus, where it is necessary we will use appropriate
	multi-indices.
\end{rem}

It is also important to note that the trajectories $x(k)$ and $u(k)$
are of course not independent. Since \eqref{eq:System_Eq} must hold
at every time step $k$, it is obvious that all forward-shifts $x(k+j)$
with $j\geq1$ of the state variables are determined by $x(k)$ and
forward-shifts $u(k+j-1)$, $j\geq1$ of the input variables, i.e.
\begin{align}
\begin{aligned}x(k+1) & =f(x(k),u(k))\\
x(k+2) & =f(f(x(k),u(k)),u(k+1))\,.\\
& \hphantom{a}\vdots
\end{aligned}
\label{eq:forward_restrictions}
\end{align}
Thus, the forward-shifts of the state variables in \eqref{eq:flat_output_xu_redundant}
are redundant. A similar argument holds for the backward-direction.
Since \eqref{eq:System_Eq} meets the submersivity condition \eqref{eq:submersivity},
there always exist $m$ functions $g(x,u)$ such that the map
\begin{align}
\begin{aligned}x^{+} & =f(x,u)\,,\end{aligned}
& \zeta=g(x,u)\label{eq:ex_diffeo}
\end{align}
is locally a diffeomorphism and hence invertible\footnote{It should be noted that the choice of $g(x,u)$ is not unique. For
	systems with $\Rank(\partial_{x}f)=n$, we could always choose $g(x,u)=u$
	and the variable $\zeta$ would represent the inputs $u$.}. If we denote by $(x,u)=\psi(x^{+},\zeta)$ its inverse
\begin{align}
\begin{aligned}x & =\psi_{x}(x^{+},\zeta)\,,\end{aligned}
& u=\psi_{u}(x^{+},\zeta)\,,\label{eq:inverse_ex_diffeo}
\end{align}
then all backward-shifts $x(k-j)$ and $u(k-j)$ of the state- and
input variables with $j\geq1$ are uniquely determined by $x(k)$
and the backward-shifts $\zeta(k-j)$, $j\geq1$ of the system variables
$\zeta$ defined by \eqref{eq:ex_diffeo}. This can be seen immediately
by a repeated evaluation of \eqref{eq:inverse_ex_diffeo}, which yields
\begin{align}
(x(k-1),u(k-1)) & =\psi(x(k),\zeta(k-1))\nonumber \\
(x(k-2),u(k-2)) & =\psi(\psi(x(k),\zeta(k-1)),\zeta(k-2))\,.\nonumber \\
& \hphantom{a}\vdots\label{eq:backward_restrictions}
\end{align}
Thus, with (\ref{eq:forward_restrictions}) and (\ref{eq:backward_restrictions})
the map \eqref{eq:flat_output_xu_redundant} can be written as
\begin{equation}
y(k)=\varphi(\zeta(k-q_{1}),\dots,\zeta(k-1),x(k),u(k),\dots,u(k+q_{2}))\,.\label{eq:flat_output}
\end{equation}
We conclude that in the trajectory-based approach the flat output
\eqref{eq:flat_output} is not only a function of $x$, $u$ and forward-shifts
of $u$, but also a function of backward-shifts of $\zeta$. Thus,
it extends the usual definition. 
\begin{rem}
	\label{rem:choice_zeta}It is important to emphasize that the flatness
	of the system (\ref{eq:System_Eq}) does not depend on the choice
	of the functions $g(x,u)$. Only the representation \eqref{eq:flat_output}
	of the flat output may differ, while the parameterization (\ref{eq:parameterizing_map_forward_backward})
	of $x$ and $u$ is not affected. If we would restrict ourselves to
	sampled data systems with $\Rank(\partial_{x}f)=n$, we could always
	choose $g(x,u)=u$. This approach leads to a definition of flatness
	as proposed in \cite{GuillotMillerioux2020}, where the flat output
	is a function of $x$, $u$, and forward- and backward-shifts of $u$.
\end{rem}

Before we give a precise geometric definition of flatness, we also
want to mention that considering both forward- and backward-shifts
in the parameterizing map \eqref{eq:parameterizing_map_forward_backward}
is actually not necessary. Indeed, if there exists a parameterizing
map \eqref{eq:parameterizing_map_forward_backward} and a flat output
\eqref{eq:flat_output}, then one can always define a new flat output
as the $r_{1}$-th backward-shift of the original flat output.\footnote{Note that the number of required backward-shifts may differ for the
	individual $m$ components of $y$, see Remark \ref{rem:individual_shifts}.} The corresponding parameterizing map is then of form
\begin{equation}
(x(k),u(k))=F(y(k),\dots,y(k+r))\label{eq:parameter_forward}
\end{equation}
with $r=r_{1}+r_{2}$.\footnote{Similarly, we may define a new flat output as the $q_{1}$-th forward-shift
	of the original flat output. The resulting flat output is then of
	the form $y(k)=\varphi(x(k),u(k),\dots,u(k+q))$, with $q=q_{1}+q_{2}$,
	and the corresponding parameterizing map of the form \eqref{eq:parameterizing_map_forward_backward}.} Thus, without loss of generality, in the remainder of the paper
we assume that the parameterizing map \eqref{eq:parameterizing_map_forward_backward}
is of the form \eqref{eq:parameter_forward} and contains only forward-shifts.

\subsection{Geometric Approach}

In order to give a concise definition of flatness including backward-shifts,
we use a space with coordinates $(\dots\zeta_{[-1]},x,u,u_{[1]}\dots)$,
where the subscript denotes the corresponding shift. Because of \eqref{eq:forward_restrictions}
and \eqref{eq:backward_restrictions}, every point of this space corresponds
to a unique trajectory $(x(k),u(k))$ of the system \eqref{eq:System_Eq}.
In accordance with (\ref{eq:ex_diffeo}), we have a forward-shift
operator $\delta$ defined by the rule
\begin{align*}
\begin{aligned}\delta(h(\dots,\zeta_{[-2]},\zeta_{[-1]},x,u,u_{[1]},\dots))=\hphantom{aaaaaaaaaa}\\
h(\dots,\zeta_{[-1]},g(x,u),f(x,u),u_{[1]},u_{[2]},\dots)
\end{aligned}
\end{align*}
for an arbitrary function $h$. Because of (\ref{eq:inverse_ex_diffeo}),
its inverse is given by the backward-shift operator
\[
\begin{aligned}\delta^{-1}(h(\dots,\zeta_{[-1]},x,u,u_{[1]},u_{[2]},\dots))=\hphantom{aaaaaaaaaa}\\
h(\dots,\zeta_{[-2]},\psi_{x}(x,\zeta_{[-1]}),\psi_{u}(x,\zeta_{[-1]}),u,u_{[1]},\dots)\,.
\end{aligned}
\]
Likewise, every point of a space with coordinates $(\ldots,y_{[-1]},y,y_{[1]},\dots)$
corresponds to a unique trajectory $y(k)$ of a trivial system. Here
the shift operators have the simple form
\begin{align*}
\delta_{y}(H(\ldots,y_{[-1]},y,y_{[1]},\dots)) & =H(\ldots,y,y_{[1]},y_{[2]},\dots)\,,\\
\delta_{y}^{-1}(H(\ldots,y_{[-1]},y,y_{[1]},\dots)) & =H(\ldots,y_{[-2]},y_{[-1]},y,\dots)\,.
\end{align*}
and $\beta$-fold application of $\delta$ and $\delta_{y}$ or their
inverses will be denoted by $\delta^{\beta}$ and $\delta_{y}^{\beta}$,
respectively.

With these preliminaries, we can give a geometric characterization
for the trajectory-based approach to discrete-time flatness suggested
in Section \ref{subsec:Equivalence-of-Solutions}. In accordance with
the literature on static and dynamic feedback linearization for discrete-time
systems, we consider a suitable neighborhood of an equilibrium $x_{0}=f(x_{0},u_{0})$,
see e.g. \cite{NijmeijervanderSchaft:1990} or \cite{Aranda-BricaireMoog:2008}.
However, we want to emphasize that for many systems the concept may
be useful even if the conditions fail to hold at an equilibrium.
\begin{defn}
	\label{def:Flatness}The system \eqref{eq:System_Eq} is said to be
	flat around an equilibrium $(x_{0},u_{0})$, if the $n+m$ coordinate
	functions $x$ and $u$ can be expressed locally by an $m$-tuple
	of functions
	\begin{equation}
	y^{j}=\varphi^{j}(\zeta_{[-q_{1}]},\dots,\zeta_{[-1]},x,u,\dots,u_{[q_{2}]})\,,\label{eq:flat_output_defintion}
	\end{equation}
	$j=1,\ldots,m$ and their forward-shifts
	\begin{align}
	\begin{aligned}y_{[1]} & =\delta(\varphi(\zeta_{[-q_{1}]},\dots,\zeta_{[-1]},x,u,\dots,u_{[q_{2}]}))\\
	y_{[2]} & =\delta^{2}(\varphi(\zeta_{[-q_{1}]},\dots,\zeta_{[-1]},x,u,\dots,u_{[q_{2}]}))\\
	& \hphantom{a}\vdots
	\end{aligned}
	\label{eq:to_be_replaced}
	\end{align}
	up to some finite order. The $m$-tuple \eqref{eq:flat_output_defintion}
	is called a flat output.
\end{defn}

If \eqref{eq:flat_output_defintion} is a flat output, then the representation
of $x$ and $u$ by the flat output is unique and a submersion of
the form\footnote{The multi-index $R=(r_{1},\dots,r_{m})$ of \eqref{eq:geo_param_map_2}
	contains the number of forward-shifts of each component of the flat
	output which is needed to express $x$ and $u$. The abbreviation
	$y_{[R]}$ denotes the components $y_{[R]}=(y_{[r_{1}]}^{1},\dots,y_{[r_{m}]}^{m})$,
	and the integer $r$ indicates the maximum number of forward-shifts
	that appear in the parameterization \eqref{eq:geo_param_map_2}, i.e.
	$r=\max(r_{1},\dots,r_{m})$.}
\begin{align}
\begin{aligned}x^{i} & =F_{x}^{i}(y,\dots,y_{[R-1]})\,,\quad i=1,\dots,n\\
u^{j} & =F_{u}^{j}(y,\dots,y_{[R]})\,,\hphantom{ia}\quad j=1,\dots,m\,.
\end{aligned}
\label{eq:geo_param_map_2}
\end{align}
We only sketch the proof of this statement. Since $x$ and $u$ can
be expressed by $\varphi,\delta(\varphi),\dots,\delta^{r}(\varphi)$,
also all forward-shifts of $u$ and all backward-shifts of $\zeta$
can be expressed by $\dots,\delta^{-1}(\varphi),\varphi,\delta(\varphi),\dots$.
By using the fact that the coordinate functions $u,u_{[1]},\dots$
and $\zeta_{[-1]},\zeta_{[-2]},\dots$ are functionally independent,
it can be shown with basic geometric concepts that also all forward-
and backward-shifts of $\varphi$ must be functionally independent.
The functional independence of $\dots,\delta^{-1}(\varphi),\varphi,\delta(\varphi),\dots$
guarantees that \eqref{eq:geo_param_map_2} is unique. Based on the
identity $(x,u)=F(\varphi,\delta(\varphi),\dots,\delta^{r}(\varphi))$
it can be shown that \eqref{eq:geo_param_map_2} is a submersion.
The fact that the Jacobian matrix $\partial_{(x,u)}F(\varphi,\delta(\varphi),\dots,\delta^{r}(\varphi))$
results in an identity matrix implies that the rows of the Jacobian
matrix of $F$ with respect to $(y,\dots,y_{[R]})$ are linearly independent.
The special structure that $F_{x}$ is independent of $y_{[R]}$ is
a consequence of the identity $\delta_{y}(F_{x})=f(F_{x},F_{u})$.

If we restrict ourselves to forward-shifts in the flat output, then
Definition \ref{def:Flatness} leads to the special case of forward-flatness.
\begin{defn}
	\label{def:forward-flatness}The system \eqref{eq:System_Eq} is said
	to be forward-flat, if it meets the conditions of Definition \ref{def:Flatness}
	with a flat output of the form $y^{j}=\varphi^{j}(x,u,\dots,u_{[q_{2}]})$.
\end{defn}

The class of forward-flat systems has already been analyzed in detail
in the literature, see e.g. \cite{KaldmaeKotta:2013}, \cite{Sira-RamirezAgrawal:2004}
and \cite{KolarKaldmaeSchoberlKottaSchlacher:2016}. In \cite{KolarSchoberlDiwold:2019},
it has been shown that every forward-flat system can be decomposed
into a smaller dimensional forward-flat subsystem and an endogenous
dynamic feedback by a suitable state- and input-transformation. Thus,
a repeated decomposition allows to check whether a system is forward-flat
or not. In \cite{KolarDiwoldSchoberl:2019}, this test has been formulated
in terms of certain sequences of distributions, similar to the test
for static-feedback linearizability in \cite{NijmeijervanderSchaft:1990}.
Thus, the property of forward-flatness can be checked in a computationally
efficient way. For flat systems that are not forward-flat, the decomposition
procedure as stated in \cite{KolarSchoberlDiwold:2019} necessarily
fails in one step, likewise the test as proposed in \cite{KolarDiwoldSchoberl:2019}.

In the following, we present a simple academic example that is flat
according to Definition \ref{def:Flatness} but not forward-flat.
In fact, the test for forward-flatness stated in \cite{KolarDiwoldSchoberl:2019}
fails already in the first step. Hence, the example already shows
that the class of forward-flat systems is indeed a strict subset of
the class of flat systems.
\begin{example}
	\label{ex:Brocket-1}Consider the system
	\begin{align}
	\begin{aligned}x^{1,+} & =u^{1}\\
	x^{2,+} & =u^{2}\\
	x^{3,+} & =x^{3}+x^{1}u^{2}+x^{2}u^{1}\,.
	\end{aligned}
	\label{eq:brocket}
	\end{align}
	With the choice $\zeta^{j}=g^{j}(x,u)=x^{j}$ for $j=1,2$, the combined
	map \eqref{eq:ex_diffeo} forms a diffeomorphism and we claim that
	the system has a flat output of the form 
	\begin{equation}
	y=(\zeta_{[-1]}^{1},x^{3}-x^{2}\zeta_{[-1]}^{1})\,.\label{eq:Flat_Output_Example_2}
	\end{equation}
	In order to prove that the system is flat, we need to show that $x$
	and $u$ can be expressed by \eqref{eq:Flat_Output_Example_2} and
	its forward-shifts. A repeated application of the shift operators
	to \eqref{eq:Flat_Output_Example_2} yields the set of equations
	\begin{align*}
	\begin{aligned}y^{1} & =\zeta_{[-1]}^{1}\,,\\
	y_{[1]}^{1} & =x^{1}\,,\\
	y_{[2]}^{1} & =u^{1}\,,\\
	y_{[3]}^{1} & =u_{[1]}^{1}\,,
	\end{aligned}
	& \begin{aligned}\hphantom{aaa}y^{2} & =x^{3}-x^{2}\zeta_{[-1]}^{1}\,,\\
	y_{[1]}^{2} & =x^{3}+x^{2}u^{1}\,,\\
	y_{[2]}^{2} & =x^{3}+x^{2}u^{1}+u^{2}(x^{1}+u_{[1]}^{1})\,,\\
	\\
	\end{aligned}
	\end{align*}
	which can be solved for $x^{1},x^{2},x^{3},u^{1},u^{2},\zeta_{[-1]}^{1}$
	and $u_{[1]}^{1}$,
	\begin{align}
	\begin{aligned}x^{1} & =y_{[1]}^{1}\,, & u^{1} & =y_{[2]}^{1}\,, & \zeta_{[-1]}^{1} & =y^{1}\,,\\
	x^{2} & =\tfrac{y_{[1]}^{2}-y^{2}}{y_{[2]}^{1}+y^{1}}\,, & u^{2} & =\tfrac{y_{[2]}^{2}-y_{[1]}^{2}}{y_{[3]}^{1}+y_{[1]}^{1}}\,, & u_{[1]}^{1} & =y_{[3]}^{1}\vphantom{\tfrac{y_{[1]}^{2}-y^{2}}{y_{[2]}^{1}+y^{1}}}\,.\\
	x^{3} & =\tfrac{y^{1}y_{[1]}^{2}+y_{[2]}^{1}y^{2}}{y^{1}+y_{[2]}^{1}}\,,
	\end{aligned}
	\label{eq:Param_Map_Example_2}
	\end{align}
	Hence, the system \eqref{eq:brocket} is flat with a flat output \eqref{eq:Flat_Output_Example_2}
	and the corresponding parameterization \eqref{eq:geo_param_map_2}
	contained in \eqref{eq:Param_Map_Example_2}.
\end{example}

We conclude this section with the following result for the special
case of single-input systems.
\begin{thm}
	\label{thm:siso}For single-input systems \eqref{eq:System_Eq} with
	$m=1$, the properties flatness, forward-flatness, and static feedback
	linearizability are equivalent.
\end{thm}

\begin{proof}
	The implication static feedback linearizability $\Rightarrow$ forward-flatness
	$\Rightarrow$ flatness follows directly from the corresponding definitions.
	For the other direction, consider a general flat output 
	\begin{equation}
	y=\varphi(\zeta_{[-q_{1}]}^{1},\dots,\zeta_{[-1]}^{1},x,u^{1},\dots,u_{[q_{2}]}^{1})\label{eq:siso_flat_output}
	\end{equation}
	of a system with $m=1$ input. Since the forward-shifts of \eqref{eq:siso_flat_output}
	are independent of $\zeta_{[-q_{1}]}^{1}$, \eqref{eq:siso_flat_output}
	would be the only function in the parameterization \eqref{eq:geo_param_map_2}
	depending on this variable. Thus, $\zeta_{[-q_{1}]}^{1}$ could not
	cancel out and accordingly \eqref{eq:siso_flat_output} itself must
	not be present in the parameterization \eqref{eq:geo_param_map_2}.
	Repeating this argumentation shows that \eqref{eq:geo_param_map_2}
	can only contain forward-shifts of \eqref{eq:siso_flat_output} which
	are already independent of $\zeta_{[-q_{1}]}^{1},\dots,\zeta_{[-1]}^{1}$.
	However, the first such forward-shift of \eqref{eq:siso_flat_output}
	is obviously a forward-flat output 
	\begin{equation}
	y=\varphi(x,u^{1},\dots,u_{[q_{2}]}^{1})\,.\label{eq:siso_forward_output}
	\end{equation}
	A similar argumentation shows that \eqref{eq:siso_forward_output}
	can actually only depend on $x$, and $u^{1}$ only appears in the
	$n$-th forward-shift. Otherwise, the forward-shifts of u could not
	cancel out and a parameterization \eqref{eq:geo_param_map_2} would
	not be possible. Thus, \eqref{eq:siso_forward_output} is a linearizing
	output in the sense of static feedback linearizability.
\end{proof}
\begin{rem}
	With Theorem \ref{thm:siso}, the question whether flatness is preserved
	under exact discretization can be reduced to the question whether
	static feedback linearizability is preserved for single-input systems.
	However, as shown in \cite{Grizzle:1986} by a counterexample, this
	is in general not true. A practical nonlinear system which remains
	flat under exact discretization is e.g. the wheeled mobile robot discussed
	in \cite{Aranda-BricaireMoog:2008}.
\end{rem}

\section{Trajectory Planning and Dynamic Feedback Linearization} \label{sec:Feedback-Linearization}In
this section, we show that flat systems (in the extended sense with
backward-shifts) still allow straightforward trajectory planning and
dynamic feedback linearization.

\subsection{Trajectory Planning\label{subsec:Trajectory-Planning}}

The popularity of differentially flat systems is mainly due to the
fact that the knowledge of a flat output allows an elegant solution
to motion planning problems. In this section, we show that also discrete-time
flat systems according to Definition \ref{def:Flatness} allow a straightforward
planning of trajectories.

Usually the motion planning problem consists in finding trajectories
$(x(k),u(k))$ that satisfy the system equations \eqref{eq:System_Eq}
and some initial and final conditions
\begin{align*}
\begin{aligned}(x(k_{i}),u(k_{i})) & =(x_{i},u_{i})\,,\end{aligned}
& (x(k_{f}),u(k_{f}))=(x_{f},u_{f})\,,
\end{align*}
with $k_{f}>k_{i}$. For flat systems, this task can be formulated
in terms of trajectories $y(k)$ for the flat output. Since every
trajectory $y(k)$ corresponds to a solution of \eqref{eq:System_Eq},
it remains to require that the trajectory $y(k)$ meets
\begin{align}
\begin{aligned}(x_{i},u_{i}) & =F(y(k_{i}),y(k_{i}+1),\dots,y(k_{i}+r))\\
(x_{f},u_{f}) & =F(y(k_{f}),y(k_{f}+1),\dots,y(k_{f}+r))\,.
\end{aligned}
\label{eq:initial}
\end{align}
If we assume that $k_{f}>k_{i}+r$ holds, then since the parameterization
\eqref{eq:geo_param_map_2} is a submersion, the set of equations
\eqref{eq:initial} can be solved independently for $2(n+m)$ values
of $y(k_{i}),\dots,y(k_{f}+r)$.\footnote{For certain parameterizations the assumption $k_{f}>k_{i}+r$ may
	be relaxed. It would be sufficient to require that the integer $k_{f}$
	is large enough, such that \eqref{eq:initial} can still be solved
	for arbitrary $2(n+m)$ values of the set $y(k_{i}),\dots,y(k_{f}+r)$.} The remaining values of $y(k_{i}),\dots,y(k_{f}+r)$ can be chosen
arbitrarily, and thus the trajectories $y(k)$ are in general not
unique.\footnote{Like in the continuous-time case, this property can be very beneficial
	in optimal control problems, e.g. minimizing control effort.} Once the trajectories $y(k)$ are determined, the corresponding state-
and input-trajectories are also uniquely determined by
\begin{align*}
(x(k),u(k))=F(y(k),y(k+1),\dots,y(k+r))\,,
\end{align*}
for $k=k_{i},\dots,k_{f}$. Since this procedure allows to connect
any two points of the state space (locally, where the system is flat),
we immediately get the following result.
\begin{thm}
	\label{thm:reachability}Flat systems according to Definition \ref{def:Flatness}
	are locally reachable.
\end{thm}

With Theorem \ref{thm:reachability} and the fact that every reachable
linear system can be transformed into Brunovsky normal form, we get
the following corollary.
\begin{cor}
	For linear time-invariant systems the properties flatness, forward-flatness
	and static feedback linearizability are equivalent.
\end{cor}

To illustrate the practical applicability of discrete-time flatness,
in the following we present a simulation result for the sampled-data
model of an induction motor. Similar to \cite{MartinRouchon:1996},
we compute a feedforward control which transfers the rotor speed between
two stationary set-points. However, instead of the classical approach
to sample and hold a feedforward control obtained from the continuous-time
model, we directly compute a discrete-time feedforward control based
on an implicit Euler-discretization of the system.
\begin{example}
	We consider the reduced-order continuous-time model of an induction
	motor discussed in \cite{Chiasson:1998}, with the state $x=(\omega,\psi_{d},\rho)$,
	the input $u=(i_{d},i_{q})$, and the same constant values $\mu,\tau_{L},J,\eta,M,n_{p}$
	as in \cite{MartinRouchon:1996}. It is well-known that the continuous-time
	system possesses a flat output which consists of the rotor speed $\omega$
	and the flux angle $\rho$. Based on an implicit Euler-discretization
	given by
	\begin{equation}
	\begin{aligned}\tfrac{1}{T_{s}}(x^{1,+}-x^{1}) & =\mu x^{2,+}u^{2}-\tfrac{\tau_{L}}{J}\\
	\tfrac{1}{T_{s}}(x^{2,+}-x^{2}) & =-\eta x^{2,+}+\eta Mu^{1}\\
	\tfrac{1}{T_{s}}(x^{3,+}-x^{3}) & =n_{p}x^{1,+}+\eta M\tfrac{u^{2}}{x^{2,+}}
	\end{aligned}
	\label{eq:impl_euler}
	\end{equation}
	with sampling time $T_{s}$, a discrete-time system \eqref{eq:System_Eq}
	can be derived by solving \eqref{eq:impl_euler} for $x^{1,+},x^{2,+}$
	and $x^{3,+}$. The obtained system is flat in the sense of Definition
	\ref{def:Flatness}, and with the choice $\zeta^{j}=g^{j}(x,u)=u^{j}$
	for $j=1,2$, a flat output is given by
	\[
	y=(x^{1}+T_{s}(\tfrac{\tau_{L}}{J}-\mu x^{2}\zeta_{[-1]}^{2}),x^{3}-T_{s}(n_{p}x^{1}+\tfrac{M\eta\zeta_{[-1]}^{2}}{x^{2}}))\,.
	\]
	This flat output has the beneficial property $y_{[1]}=(x^{1},x^{3})$,
	i.e., its first forward-shift coincides with the continuous-time flat
	output. From the corresponding parametrization \eqref{eq:geo_param_map_2},
	a discrete-time feedforward control has been computed that transfers
	the rotor speed between two stationary set-points like in \cite{MartinRouchon:1996}.
	Applying the calculated feedforward control (piecewise constant during
	the sampling intervals) to the continuous-time system yields the simulation
	result shown in Fig. \ref{fig:Open-loop-simulation-results}. It can
	be observed that the reference trajectory is perfectly tracked.
	\begin{figure}
		\centering
		
		\includegraphics[width=1\columnwidth]{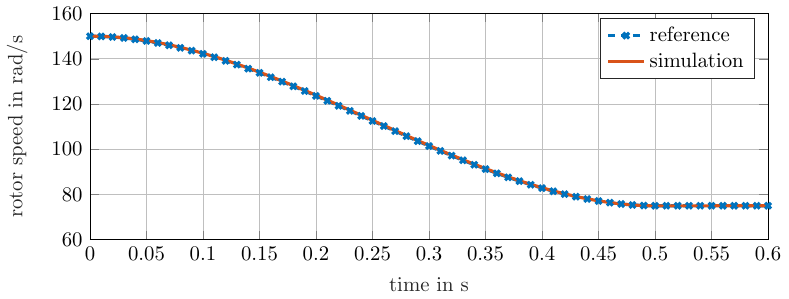}\caption{\label{fig:Open-loop-simulation-results}Open-loop simulation result
			($T_{s}=10\mathrm{ms}$, as in \cite{MartinRouchon:1996}).}
	\end{figure}
\end{example}

\subsection{Dynamic Feedback Linearization}

In the continuous-time framework, flatness is closely related to the
dynamic feedback linearization problem. To be precise, the class of
differentially flat systems is equivalent to the class of systems
linearizable via endogenous dynamic feedback. A continuous-time dynamic
feedback $\dot{z}=\alpha(x,z,v)$ with $u=\beta(x,z,v)$ is said to
be endogenous, if there exists a one-to-one correspondence between
trajectories of the closed-loop system and trajectories of the original
system. As a consequence, $z$ and $v$ can be expressed as functions
of $x$, $u$ and time derivatives of $u$.

According to \cite{Aranda-BricaireMoog:2008}, a discrete-time dynamic
feedback is said to be endogenous, if its states $z$ and inputs $v$
can be expressed as functions of $x$, $u$ and forward-shifts of
$u$. It can be shown that the class of discrete-time systems that
is linearizable via endogenous dynamic feedback in the sense of \cite{Aranda-BricaireMoog:2008}
exactly corresponds to the class of forward-flat systems. In the following,
we show that also for flat systems according to Definition \ref{def:Flatness}
there always exists a linearizing discrete-time dynamic feedback.
However, in general the required feedback is not contained within
the class of endogenous dynamic feedbacks proposed in \cite{Aranda-BricaireMoog:2008}.
\begin{thm}
	\label{thm:dynamic_feedback_linearization}A flat system \eqref{eq:System_Eq}
	can be linearized by a dynamic feedback
	\begin{align}
	\begin{aligned}z^{+} & =\alpha(x,z,v)\,,\end{aligned}
	\hphantom{a} & u=\beta(x,z,v)\label{eq:general_dynamic_feedback}
	\end{align}
	with the following properties:
\end{thm}

\begin{enumerate}
	\item[(a)] The closed-loop system is submersive.
	\item[(b)] The trajectories of the closed-loop system are in one-to-one correspondence
	to the trajectories of the original system.
\end{enumerate}
\begin{proof}
	The fact that the parameterizing map \eqref{eq:geo_param_map_2} is
	a submersion implies that also the parameterization $F_{x}$ is a
	submersion. Consequently, there exists a map $z=F_{z}(y,\dots,y_{[R-1]})$,
	such that the combined map $(x,z)=(F_{x},F_{z}):=F_{xz}$ forms a
	diffeomorphism, with $\dim(z)=p\leq mr-n$.  We define the map $\Phi(y,\dots y_{[R]})$
	given by
	\begin{equation}
	\begin{aligned}x & =F_{x}(y,\dots,y_{[R-1]})\,, & v & =y_{[R]}\,,\\
	z & =F_{z}(y,\dots,y_{[R-1]})
	\end{aligned}
	\label{eq:phi}
	\end{equation}
	and its inverse $\hat{\Phi}(x,z,v)$ given by
	\begin{align}
	\begin{array}{c}
	\begin{aligned}(y,\dots,y_{[R-1]}) & =\hat{F}_{xz}(x,z)\,,\end{aligned}
	\end{array} & y_{[R]}=v\,.\label{eq:phi_d}
	\end{align}
	Based on \eqref{eq:phi_d},  a linearizing dynamic feedback is given
	by
	\begin{align}
	\begin{aligned}z^{+} & =\delta_{y}(F_{z})\circ\hat{\Phi}(x,z,v)\,,\end{aligned}
	& u=F_{u}\circ\hat{\Phi}(x,z,v)\,,\label{eq:dynamic_feedback}
	\end{align}
	as we prove next by transforming the closed-loop dynamics
	\begin{align}
	\begin{aligned}x^{+} & =f(x,F_{u}\circ\hat{\Phi}(x,z,v))\\
	z^{+} & =\delta_{y}(F_{z})\circ\hat{\Phi}(x,z,v)
	\end{aligned}
	\label{eq:closed_loop_deriv}
	\end{align}
	into Brunovsky normal form. With the state-transformation $(x,z)=F_{xz}(y,\dots,y_{[R-1]})$
	and the input-transformation $v=y_{[R]}$ we get
	\[
	(y^{+},\dots,y_{[R-1]}^{+})=\hat{F}_{xz}\circ\left[\begin{array}{c}
	f(x,F_{u}\circ\hat{\Phi}(x,z,v))\\
	\delta_{y}(F_{z})\circ\hat{\Phi}(x,z,v)
	\end{array}\right]\circ\Phi\,,
	\]
	which can be rewritten as
	\begin{align}
	\begin{aligned}(y^{+},\dots,y_{[R-1]}^{+}) & =\hat{F}_{xz}\circ\left[\begin{array}{c}
	f(F_{x},F_{u})\\
	\delta_{y}(F_{z})
	\end{array}\right]\,.\end{aligned}
	\label{eq:zw_dyn}
	\end{align}
	Since the parameterization \eqref{eq:geo_param_map_2} satisfies the
	system equations identically, by substituting $F$ into $\delta(x)=f(x,u)$
	we get the relation $\delta_{y}(F_{x})=f(F_{x},F_{u})$ and may rewrite
	\eqref{eq:zw_dyn} as
	\[
	(y^{+},\dots,y_{[R-1]}^{+})=\hat{F}_{xz}\circ\left[\begin{array}{c}
	\delta_{y}(F_{x})\\
	\delta_{y}(F_{z})
	\end{array}\right]\,.
	\]
	Due to $\hat{F}_{xz}\circ(\delta_{y}(F_{x}),\delta_{y}(F_{z}))=\delta_{y}(\hat{F}_{xz}\circ(F_{x},F_{z}))$,
	and since per definition $\hat{F}_{xz}\circ(F_{x},F_{z})$ yields
	identically $(y,\dots,y_{[R-1]})$, the Brunovsky normal form follows
	as\footnote{The multi-index $R=(r_{1},\dots,r_{m})$ denotes the length of the
		individual chains $(y^{j,+},\dots,y_{[r_{j}-1]}^{j,+})=(y_{[1]}^{j},\dots,y_{[r_{j}]}^{j})$
		of the Brunovsky normal form. For flat systems \eqref{eq:System_Eq}
		with $\Rank(\partial_{u}f)<m$, redundant inputs can be chosen as
		components of the flat output, and the Brunovsky normal form of the
		corresponding extended system has chains of length zero.}
	\[
	(y^{+},\dots,y_{[R-1]}^{+})=\delta_{y}(y,\dots,y_{[R-1]})=(y_{[1]},\dots,y_{[R]})\,.
	\]
	Since the closed-loop system can be transformed into Brunovsky normal
	form, the dynamic feedback \eqref{eq:dynamic_feedback} preserves
	both submersivity and reachability, and it remains to show condition
	(b). Due to \eqref{eq:phi} we have a one-to-one correspondence between
	trajectories of the closed-loop system and trajectories of the trivial
	system. However, the trajectories of the trivial system are by the
	definition of flatness in one-to-one correspondence to the trajectories
	of the original system, which completes the proof.
\end{proof}
Since two submersive systems \eqref{eq:System_Eq} with a one-to-one
correspondence between their trajectories are either both flat or
non-flat, we get the following corollary.
\begin{cor}
	Applying a discrete-time dynamic feedback \eqref{eq:general_dynamic_feedback}
	with the properties (a) and (b) of Theorem \ref{thm:dynamic_feedback_linearization}
	does not affect the flatness of a system \eqref{eq:System_Eq}.
\end{cor}

In contrast to a continuous-time endogenous dynamic feedback, the
additional condition (a) is required. Otherwise, the reachability
and hence also flatness could be lost. The difference to the notion
of discrete-time endogenous dynamic feedback introduced in \cite{Aranda-BricaireMoog:2008}
is that in our case the variables $z$ and $v$ of \eqref{eq:general_dynamic_feedback}
may depend on both forward- and backward-shifts $(\dots,\zeta_{[-1]},x,u_{[1]},\dots)$
of the system variables.
\begin{rem}
	\label{rem:Dynamic_feedback_lin_equivalent_flatness}In the classical
	dynamic feedback linearization problem, the one-to-one correspondence
	between trajectories of the closed-loop system and the original system
	is not required. Thus, the linearizing output of the closed-loop system
	can possibly not be expressed in terms of forward- and backward-shifts
	of the original system variables.
\end{rem}

\section{Conclusion}

In this contribution, we have investigated the extension of the notion
of discrete-time flatness to both forward- and backward-shifts. We
have shown that adding backward-shifts fits very nicely with the concept
of one-to-one correspondence of solutions of the original system and
a trivial system, as it is well-known from the continuous-time case.
Even with backward-shifts, reachability and controllability still
hold and trajectories can be planned in a straightforward way. Furthermore,
such systems can be linearized by a particular subclass of dynamic
feedbacks. Thus, from an application point of view, the basic properties
of forward-flat systems are preserved. Since we expect that the class
of flat systems in the extended sense including backward-shifts is
significantly larger than the class of forward-flat systems, this
opens many new perspectives for practical applications as illustrated
by the presented induction motor. Future research will deal with the
systematic construction of flat outputs and finding necessary and/or
sufficient conditions as they already exist for forward-flat systems.
Another open question, which is motivated by the continuous-time case,
is whether the class of flat systems is only a subset of or equivalent
to the class of systems linearizable by dynamic feedback, see Remark
\ref{rem:Dynamic_feedback_lin_equivalent_flatness}.

\bibliographystyle{IEEEtran}
\bibliography{IEEEabrv,Bibliography_Johannes_April_2020}

\end{document}